\documentclass[11pt]{amsart}
\usepackage{amsmath,amscd,amssymb,amsthm,mathtools}
\usepackage[mathscr]{eucal}
\usepackage[frame,cmtip,arrow,matrix,line,graph,curve]{xy}
\usepackage{graphpap, color}
\usepackage{cancel}
\usepackage{verbatim}
\usepackage{adjustbox}
\usepackage[margin=1in]{geometry}
\usepackage{pgfplots}
\pgfplotsset{compat=1.18}

\usepackage{tikz}
\usepackage{float}
\usepackage{booktabs}
\usepackage{multirow}
\usepackage{tabularx}
\usepackage{enumitem}
\usepackage{microtype}

\usepackage{hyperref}

\numberwithin{equation}{section}

\newtheorem{theo}{Theorem}[section]
\newtheorem{coro}[theo]{Corollary}
\newtheorem{prop}[theo]{Proposition}
\newtheorem{lemm}[theo]{Lemma}

\theoremstyle{definition}
\newtheorem{defi}[theo]{Definition}
\newtheorem{ques}[theo]{Question}
\newtheorem{exam}[theo]{Example}

\newcommand{\CC}{\mathbb{C}}

\newcommand{\PP}{\mathbb{P}}
\newcommand{\QQ}{\mathbb{Q}}

\newcommand{\ZZ}{\mathbb{Z}}

\newcommand{\HH}{\mathbb{H}}

\newcommand{\bH}{\mathbf{H}}

\def\cA{{\mathcal A}}
\def\cB{{\mathcal B}}
\def\cC{{\mathcal C}}

\def\cM{{\mathcal M}}

\def\cI{{\mathcal I}}

\def\begeq{\begin{equation}}
\def\endeq{\end{equation}}
\def\and{\quad{\rm and}\quad}
\def\lra{\longrightarrow }

\def\beq{\begin{equation}}
\def\eeq{\end{equation}}

\def\bcM{\overline{\cM}}
\def\cAn{\cA^{(n)}}
\def\Yn{Y^{(n)}}

\title{Hodge-Deligne and Poincar\'e polynomials of $\overline{\mathcal M}_{1,n}$} 
\author{Young-Hoon Kiem}
\address{School of Mathematics, Korea Institute for Advanced Study, 85 Hoegiro, Dongdaemun-gu, Seoul 02455, Korea}
\email{kiem@kias.re.kr}

\date{}

\begin{document}

\begin{abstract}
    By studying wall crossings of Hassett moduli spaces of weighted pointed curves, we prove a recursive formula that reduces the computation of the cohomology of $\bcM_{g,n}$ to that of the moduli space $\bcM_{g,(0^+)^n}$ of stable pointed curves of minimal weights. For $g=1$, we further prove an explicit formula for the Hodge-Deligne and Poincar\'e polynomials of $\bcM_{1,(0^+)^n}$. Combining these, we prove formulas that enable us to compute the cohomology of $\bcM_{1,n}$ efficiently.  
\end{abstract}

\maketitle

\section{Introduction}

The moduli space $\bcM_{g,n}$ of stable curves of genus $g$ with $n$ marked points is one of the most important objects in algebraic geometry \cite{DM, Knu}. It parametrizes connected projective curves $C$ of arithmetic genus $g$ with at worst nodal singularities together with $n$ ordered smooth marked points $p_1,\cdots, p_n$ in $C$ such that 
\begin{enumerate}
\item $p_i$ are all distinct and 
\item $\omega_C(\sum_i p_i)$ is ample. 
\end{enumerate}
It is a smooth proper Deligne-Mumford stack with projective coarse moduli space. It plays a fundamental role in modern enumerative geometry and serves as a model for moduli spaces of higher dimensional varieties. 
For over half a century, much effort has been exerted to understand the geometry and topology of $\bcM_{g,n}$ and there has been a constant stream of new discoveries even until now \cite{BK,BKs,CK,CKL1,CKL2,CKL3,KSY}. 


In this paper, we focus on the cohomology of $\bcM_{1,n}$.
Getzler \cite{Getzler2} found a formula for the $S_n$-equivariant Hodge-Deligne
polynomials of $\bcM_{1,n}$ for all $n$, by the theory of modular operads \cite{GK}.
It expresses them in terms of the equivariant Hodge-Deligne polynomials of the open
moduli spaces $\cM_{1,n}$, computed in \cite{Getzler3} by means of the Eichler-Shimura
isomorphism, and of genus zero data, combined by plethystic composition.
This determines the Hodge numbers of $\bcM_{1,n}$ together with the $S_n$-action on
its cohomology. Getzler listed the Hodge-Deligne polynomials of $\bcM_{1,n}$ for
$n\le 15$ \cite{Getzler2}, and the $S_n$-equivariant ones were tabulated for $n\le 20$
by Bergstr\"om \cite{BeT}.

In genus zero, on the other hand, the Poincar\'e polynomial of $\bcM_{0,n}$ is given by
an explicit recursive formula \cite[(4.12)]{Manin}, which can be proved by Keel's
blowup construction \cite{Keel}. Therefore it seems reasonable to ask the following.
\begin{ques}
Is there an explicit recursive formula for the Poincar\'e and Hodge-Deligne polynomials
of $\bcM_{1,n}$ which comes from a sequence of blowups, as in genus zero?
\end{ques}
The purpose of this paper is to provide a positive answer to this question. This gives
a new proof of Getzler's results on the Hodge numbers of $\bcM_{1,n}$, by elementary
formulas which involve only binomial coefficients, Stirling numbers and dimensions of
spaces of cusp forms, and an effective way to compute the Hodge and Betti numbers of
$\bcM_{1,n}$ for large $n$. Since our recursion comes from explicit smooth blowups, it
also holds for any invariant satisfying the blowup formula, such as Chow groups.


\medskip

Our computation of $H^*(\bcM_{1,n})$ is achieved in two stages. 
We first reduce the computation to a simpler moduli space $\bcM_{1,(0^+)^n}$, where $0^+$ denotes a positive rational number sufficiently close to $0$.  
We note that many of the strata in $\bcM_{1,n}$ disappear and the computation becomes much easier if we allow all the marked points to coincide. This happens in the moduli space $\bcM_{1,(0^+)^n}$ of stable pointed curves with weight $0^+$ for each marked point \cite{Hst}. 
By studying wall crossings in Hassett's moduli spaces of weighted pointed curves, we prove the following.
\begin{theo}[Theorem \ref{17}] Let $g\ge 1$ and fix any invariant $\bH$ of smooth proper Deligne-Mumford stacks satisfying the blowup formula \eqref{12} such as the Hodge-Deligne or Poincar\'e polynomial. 
Let $\bcM_{g,r|s}=\bcM_{g,(1^r,(0^+)^s)}$ be the moduli space of stable curves of genus $g$ with $r$ points of weight 1 and $s$ points of weight $0^+$. Let $\bH_{r|s}:=\bH_{\bcM_{g,r|s}}$. Then for $r,s\ge 0$ with $r+s>0$ and $n>0$, the following hold:
   \beq\label{32} 
   \bH_{r|s} = \bH_{r-1|s+1} + \sum_{k=0}^{s-2} \binom{s}{k}\,\bH_{r|k}  \cdot (\bH_{\PP^{s-k-1}}-1), \quad r\ge 1
   ,\eeq
   \beq\label{33} \bH_{\bcM_{g,n}}= \bH_{n|0} = \bH_{0|n} + \sum_{r=1}^{n-2} \sum_{k=0}^{s-2} \binom{s}{k}\, \bH_{r|k} \cdot (\bH_{\PP^{s-k-1}}-1)
   .\eeq
All $\bH_{r|s}$ are recursively determined by \eqref{32} if we know $\bH_{0|n}=\bH_{\bcM_{g,(0^+)^n}}$ for all $n$.   
\end{theo}
The theorem is equivalent to \cite[Theorem A]{KSY}. Our proof comes from explicit blowups and gives us a Hodge conjecture comparison (Corollary \ref{HCc}).

Next we compute the cohomology of $\bcM_{1,(0^+)^n}$ piece by piece. 
Let $$\pi:\bcM_{1,(0^+)^n}\lra \bcM_{1,0^+}=\bcM_{1,1}$$
be the morphism that forgets all the marked points except the first. The inverse image $\cI_n=\pi^{-1}(\cM_{1,1})$ of $\cM_{1,1}$ is the $(n-1)$-fold fiber product of the universal curve $$q:\cC\lra \cM_{1,1}=SL_2(\ZZ)\backslash \HH, \quad \HH=\{z\in \CC\,|\, \mathrm{Im}\, z>0\}$$ and we have
\beq\label{34} R\pi_*\QQ_{\cI_n}=(Rq_*\QQ_\cC)^{\otimes n-1}=\left((\QQ\oplus \QQ(-1)[-2])\oplus V[-1]\right)^{\otimes n-1}\eeq
by the relative K\"unneth formula, where $V=R^1q_*\QQ$.
Note that $H^*_c(\cM_{1,1},V^{\otimes k})=0$ for $k$ odd by the action of the generic stabilizer $-I\in SL_2(\ZZ)$.  
By expanding the right hand side of \eqref{34} and applying the Clebsch-Gordan rule, the computation of $H^*_c(\cI_n)$ is reduced to that of $H^*_c(\cM_{1,1}, \mathrm{Sym}^{2j}V)$ which can be obtained by the Eichler-Shimura isomorphism (cf. Lemma \ref{35}). 
Combining these, we find that the Hodge-Deligne polynomial of the interior $\cI_n$ is (cf. Proposition \ref{3})
\beq\label{36}
\sum_{j=0}^{\lfloor\frac{n-1}{2}\rfloor} \binom{n-1}{2j} (uv+1)^{n-1-2j} \left( \alpha_{jj}(uv)^{j+1} - \sum_{i=0}^{j-1} \alpha_{ij}(uv)^i\right),
\eeq
\[-\sum_{j=0}^{\lfloor\frac{n-1}{2}\rfloor} \binom{n-1}{2j} (uv+1)^{n-1-2j}  \sum_{i=0}^{j-1} \alpha_{ij}\beta_{2j-2i+2}(uv)^i(u^{2j-2i+1}+v^{2j-2i+1})  \]
where \[ \alpha_{ij}=\binom{2j}{i}-\binom{2j}{i-1}
\and 
\beta_k=\left\{\begin{matrix}  \lfloor \frac{k}{12}\rfloor -1 & \text{if } k\equiv 2 \text{ mod } 12\\
\lfloor \frac{k}{12}\rfloor & \text{if } k\not\equiv 2 \text{ mod } 12.\end{matrix}\right.
\]

The complement $\cB_n=\pi^{-1}(\infty)=\bcM_{1,(0^+)^n}-\cI_n$ consists of pointed rational nodal curves. 
Let $\cB_{n,j}\subset \bcM_{1,(0^+)^n}$ be the locus of pointed rational curves whose dual graphs are necklaces with $j$ vertices. 
It is obvious that there are $\gamma_{n,j}$
dual graphs with $j$ vertices, where
\[ \gamma_{n,j}=\left\{\begin{matrix} n\\ j\end{matrix}\right\}\frac{(j-1)!}{2} \ \ \text{for }\ \ j>2, \quad \gamma_{n,1}=1, \ \ \gamma_{n,2}=2^{n-1}-1\]
and $\left\{\begin{matrix} n\\ j\end{matrix}\right\}$ denotes the Stirling number of the second kind. 
For $j\ge 3$, each connected component of $\cB_{n,j}$ is isomorphic to $(\CC^*)^{n-j}$ and hence $H^*_c(\cB_{n,j})$ has Hodge-Deligne polynomial $\gamma_{n,j}(uv-1)^{n-j}$. For $j=1,2$, we have to consider an additional action of $S_2$ which acts on $\CC^*$ by $z\mapsto z^{-1}$. Combining all these, we find that the Hodge-Deligne polynomial of the boundary $\cB_n$ is 
\beq\label{37}
\frac{(uv+1)^{n-1}+(uv-1)^{n-1}}{2} + \gamma_{n,2} \frac{(uv+1)^{n-2}+(uv-1)^{n-2}}{2} + \sum_{j=3}^n \gamma_{n,j} (uv-1)^{n-j}.
\eeq
By \eqref{36} and \eqref{37}, we have the following.
\begin{theo} [Theorem \ref{22}]
The Hodge-Deligne polynomial $\bH_{0|n}$ of $\bcM_{1,(0^+)^n}=\bcM_{1,0|n}$ is 
\beq\label{38}
\sum_{j=0}^{\lfloor\frac{n-1}{2}\rfloor} \binom{n-1}{2j} (uv+1)^{n-1-2j} \left( \alpha_{jj}(uv)^{j+1} - \sum_{i=0}^{j-1} \alpha_{ij}(uv)^i\right),
\eeq
\[ +\frac{(uv+1)^{n-1}+(uv-1)^{n-1}}{2} + \gamma_{n,2} \frac{(uv+1)^{n-2}+(uv-1)^{n-2}}{2} + \sum_{j=3}^n \gamma_{n,j} (uv-1)^{n-j} \]
\[-\sum_{j=0}^{\lfloor\frac{n-1}{2}\rfloor} \binom{n-1}{2j} (uv+1)^{n-1-2j}  \sum_{i=0}^{j-1} \alpha_{ij}\beta_{2j-2i+2}(uv)^i(u^{2j-2i+1}+v^{2j-2i+1}).  \]
The Hodge-Deligne polynomial $\bH_{n|0}$ of $\bcM_{1,n}=\bcM_{1,n|0}$ can be computed by the recursive formula \eqref{33} together with \eqref{32} and \eqref{38}.
\end{theo}
Letting $u=v=-t$, we get the Poincar\'e polynomial of $\bcM_{1,n}$. Note that only the last line in \eqref{38} has odd degree terms. 

It is straightforward to implement the above theorem into a computer program and we can compute the Hodge-Deligne and Poincar\'e polynomials of $\bcM_{1,n}$ for large $n$ within a relatively short period of time. 

\medskip

Motivated by recent discoveries on the cohomology of $\bcM_{0,n}$ \cite{BK, BKs,CK,CKL1,CKL2,CKL3}, we raise the following. 
\begin{ques}\label{Q1.4}
(1) The sequence of blowups in Section \ref{S3} is not $S_n$-equivariant. Can the
method of this paper be refined to compute the $S_n$-equivariant Hodge-Deligne
polynomials of $\bcM_{1,n}$, giving a new proof of Getzler's formula
\cite{Getzler2, Getzler3}?


(2) Do the 
coefficients of the even degree Poincar\'e polynomial $P^{\mathrm{even}}_{\bcM_{1,n}}(x)=\sum_{i=0}^n x^i
\dim H^{2i}(\bcM_{1,n})$ of $\bcM_{1,n}$ form a log-concave or $k$-ultra log-concave sequence for some $k$?

(3) Does the probability distribution of the coefficients of $P^{\mathrm{even}}_{\bcM_{1,n}}(x)/P^{\mathrm{even}}_{\bcM_{1,n}}(1)$ converge to a Gaussian distribution? 

(4) Can we find effective estimates of the Betti numbers of $\bcM_{1,n}$ for large $n$? 

(5) Can we also use the method in this paper to calculate the Hodge-Deligne and Poincar\'e polynomials of $\bcM_{g,n}$ for $g\ge 2$? Note that Theorem \ref{17} holds for all $g\ge 1$ and hence it suffices to calculate the Hodge-Deligne polynomial of $\bcM_{g,0|n}$.
\end{ques}

In subsequent papers, we will continue our investigation to seek answers to the above questions. 

\medskip

\noindent\textbf{Convention}. 
All schemes and stacks in this paper are defined over the complex number field $\CC$
and all cohomology groups have rational coefficients. 

\medskip

\noindent\textbf{Acknowledgement}. It is my pleasure to thank Gergely B\'erczi, Jonas Bergstr\"om, Jinwon Choi, Shiyue Li and Han-Bom Moon for useful discussions and comments. 

\bigskip

\section{Moduli spaces of weighted pointed curves and Hodge-Deligne polynomials}
In this section, we collect necessary facts about Hassett moduli spaces \cite{Hst} and Hodge-Deligne polynomials that will be used throughout this paper. 

\begin{defi} \cite{Hst}
Let $g, n\ge 0$. For $\cA=(a_1,\cdots,a_n)\in (0,1]^n$ with $2g-2+\sum_i a_i>0$, a connected curve $C$ with marked points $p_1, \cdots,p_n\in C$ is called $\cA$-\emph{stable} if \begin{enumerate}
\item $C$ has at worst nodal singularities and $p_i$ are all smooth points;
\item if $p_{i_1}=p_{i_2}=\cdots=p_{i_k}$, then $a_{i_1}+\cdots+a_{i_k}\le 1$;
\item $\omega_C(a_1p_1+\cdots +a_np_n)$ is ample. 
\end{enumerate}
\end{defi}
Let $\bcM_{g,\cA}$ denote the moduli stack of $\cA$-stable $n$-pointed curves of genus $g$. 
In particular, when $a_i>1/2$ for all $i$, $\bcM_{g,\cA}$ coincides with the Knudsen-Deligne-Mumford moduli stack $\bcM_{g,n}$ of stable curves of genus $g$ with $n$ marked points \cite{DM,Knu}.

\medskip

Let us summarize useful facts about $\bcM_{g,\cA}$ from \cite{Hst}.
\begin{theo}\label{2} 
Let $\cA=(a_1,\cdots,a_n)\in (0,1]^n$ with $2g-2+\sum_{i=1}^n a_i>0$.

(1) $\bcM_{g,\cA}$ is a smooth connected proper Deligne-Mumford stack of dimension $3g-3+n$ with projective coarse moduli space.  

(2) If $\cA=(a_1,\cdots,a_n)$, $\cA'=(a'_1,\cdots, a'_n)$ and $a_i\ge a'_i$ for all $i$, then we have a natural morphism 
\beq\label{1} \rho_{\cA,\cA'}:\bcM_{g,\cA}\lra \bcM_{g,\cA'}\eeq 
which is obtained by contracting rational components in an $\cA$-stable curve which are not $\cA'$-stable. 

(3) For $2<r\le n$, if $a'_{i_1}+\cdots+a'_{i_r}\le 1$  and $a_{i_1}+\cdots+a_{i_r}>1$ but any proper subset  of $\{a_{i_1},\cdots, a_{i_r}\}$ has sum at most one while $a_j=a_j'$ for $j\ne i_1,\cdots,i_r$, then the morphism $\rho_{\cA,\cA'}$ is the smooth blowup along the locus $\bcM_{g,\bar \cA'}$ in $\bcM_{g,\cA'}$ where $p_{i_1}=\cdots=p_{i_r}$ and $\bar \cA'$ is obtained from $\cA'$ by replacing $a'_{i_1}, \cdots, a'_{i_r}$ by their sum. 

(4) If $r\le 2$ for any $I\subset [n]$ such that $a_{i_1}+\cdots+a_{i_r}>1$ and $a'_{i_1}+\cdots+a'_{i_r}\le 1$, then $\rho_{\cA,\cA'}$ is an isomorphism. 
\end{theo}
\begin{proof}
    (1) is \cite[Theorem 2.1]{Hst} and (2) is \cite[Theorem 4.1]{Hst} while (3) is \cite[Remark 4.6]{Hst} and (4) is \cite[Corollary 4.7]{Hst}. 
\end{proof}

\medskip

In this paper, we will consider invariants $\bH$ of 
smooth proper Deligne-Mumford stacks 
which assign an element $\bH_X$ of a ring to each 
$X$ over $\CC$ satisfying the \emph{blowup formula}: If $X\to Y$ is the smooth blowup along $Z$ of codimension $c$, then 
\beq\label{12}
\bH_X=\bH_Y + \bH_Z \cdot (\bH_{\PP^{c-1}}-1).\eeq 
For instance, $\bH$ may be the Hodge-Deligne or the Poincar\'e polynomial 
of $X$. Our result in \S\ref{S3} (Theorem \ref{17}) relies only on the blowup formula \eqref{12}.  

For the computation in \S\ref{S4} (Theorem \ref{22}), we will let $\bH$ be the Hodge-Deligne polynomial
$$\bH_X(u,v)=\sum_k (-1)^k\sum_{p,q}u^pv^q\dim H^{p,q;k}_c(X) \in \ZZ[u,v], \quad H^{p,q;k}_c(X)=\mathrm{gr}^p_F\mathrm{gr}_{p+q}^W H^k_c(X)$$
where $\mathrm{gr}_F^p$ and $\mathrm{gr}^W_{p+q}$ denote the graded parts with respect to the Hodge filtration and weight filtration respectively. 
Since the mixed Hodge structure is natural, the Hodge-Deligne polynomial is a motivic invariant in the sense that 
\begin{enumerate}
    \item $\bH_X=\bH_{X-Z}+\bH_Z$ for any closed $Z\subset X$ and
    \item $\bH_{X\times Y}=\bH_X\cdot \bH_Y.$
\end{enumerate}
By letting $u=v=-t$, we obtain the virtual Poincar\'e polynomial
$$P_X(t)=\bH_X(-t,-t)\in \ZZ[t]$$
of $X$ which is also a motivic invariant. 

\bigskip

\section{Reduction to minimally weighted curves}\label{S3}
In this section, we prove an explicit formula (Theorem \ref{17}) that calculates the invariants of $\bcM_{g,n}=\bcM_{g,1^n}$ from those of $\bcM_{g,(0^+)^n}$ where $(0^+)^n\in (0,1]^n$ denotes an element sufficiently close to zero, for $g\ge 1$. 

Fix $g\ge 1$ and $n\ge 1$. Then $\bcM_{g,\cA}$ is a nonempty proper smooth Deligne-Mumford stack for any $\cA\in (0,1]^n$. 
For $0\le r< n$ and $0< s\le n-r$, let
\beq\label{8} \cAn_{r,s}=(1,\cdots,1,\frac{s}{n-r},\frac1{n-r},\cdots,\frac1{n-r})\in (0,1]^n\eeq
where $1$ is repeated $r-1$ times and $\frac1{n-r}$ is repeated $n-r$ times. 
For instance, $\cAn_{0,s}=(\frac1n,\cdots,\frac1n)$ for any $s$. 
Let 
\beq\label{9} \Yn_{r,s}:=\bcM_{g,\cAn_{r,s}}.\eeq 
In particular, $\Yn_{0,n}=\bcM_{g,(0^+)^n}$ and 
\beq\label{11} \Yn_{n-2,2}=\Yn_{n-1,1}=\bcM_{g,n}.\eeq
Moreover, by Theorem \ref{2} (4),  
\beq\label{15} \Yn_{r,n-r-1}=\Yn_{r,n-r}=\bcM_{g,(1^r,(0^+)^{n-r})}=:\bcM_{g,r|n-r}\eeq
is the moduli space of stable curves of genus $g$ with $r$ heavy points and $n-r$ light points \cite{KSY}. 

Note that if $r>r'$ or $r=r'$ and $s>s'$, then $\cAn_{r,s}- \cAn_{r',s'}\ge 0$ and hence we have the natural morphism
$$\rho_{(r,s),(r',s')}:\Yn_{r,s}\lra \Yn_{r',s'}$$
by Theorem \ref{2} (2). 

Fix an invariant $\bH$ satisfying the blowup formula and let 
\beq\label{14} \bH_{r|s}:=\bH_{\bcM_{g,r|s}}=\bH_{Y^{(r+s)}_{r,s}}\eeq 
be the invariant 
of $\bcM_{g,r|s}$. 
Then we have $\bH_{0|n}=\bH_{\bcM_{g,(0^+)^n}}$ and $\bH_{n|0}=\bH_{\bcM_{g,n}}$ by \eqref{15}.

Now our goal is to reach $\Yn_{n-2,2}=\bcM_{g,n}$ starting from $\Yn_{0,n}=\bcM_{g,(0^+)^n}$ by smooth blowups so that we can reduce the computation of $\bH_{n|0}$ to $\bH_{0|n}$. Let us first reach $\Yn_{1,n-1}$. 
The natural map
$$\rho_{\cAn_{1,1},\cAn_{0,n}}:\Yn_{1,1}\lra \Yn_{0,n}$$
is the smooth blowup along the locus $\bcM_{g,1|0}=\bcM_{g,1}$ of $p_1=p_2=\cdots =p_n$ by Theorem \ref{2} (3). 
Likewise, for $1\le k\le n-3$, the natural map 
$$\rho_{\cAn_{1,k+1},\cAn_{1,k}}:\Yn_{1,k+1}\lra \Yn_{1,k}$$
is the smooth blowup along the locus of $p_1=p_{i_1}=\cdots=p_{i_{n-k-1}}$ for $\{i_1,\cdots,i_{n-k-1}\}\subset \{2,\cdots,n\}$ which consists of $\binom{n-1}{k}$ copies of $\bcM_{g,1|k}$.
Note that $\Yn_{1,n-2}=\Yn_{1,n-1}$ by \eqref{15}. 
By applying the blowup formula \eqref{12} to these morphisms, we find that 
\beq\label{13}
\bH_{1|n-1}=\bH_{0|n}+\sum_{k=0}^{n-3} \binom{n-1}{k} \bH_{1|k} (\bH_{\PP^{n-k-2}}-1).
\eeq

Next, let us reach $\Yn_{2,n-2}$ from $\Yn_{1,n-1}$. 
The natural map
$$\rho_{\cAn_{2,1},\cAn_{1,n-1}}:\Yn_{2,1}\lra \Yn_{1,n-1}$$
is the smooth blowup along the locus $\bcM_{g,2|0}=\bcM_{g,2}$ of $p_2=p_3=\cdots =p_n$ by Theorem \ref{2} (3). 
For $1\le k\le n-4$, the natural map 
$$\rho_{\cAn_{2,k+1},\cAn_{2,k}}:\Yn_{2,k+1}\lra \Yn_{2,k}$$
is the smooth blowup along the locus of $p_2=p_{i_1}=\cdots=p_{i_{n-k-2}}$ for $\{i_1,\cdots,i_{n-k-2}\}\subset \{3,\cdots,n\}$ which consists of $\binom{n-2}{k}$ copies of $\bcM_{g,2|k}$.
We have $\Yn_{2,n-3}=\Yn_{2,n-2}$ by \eqref{15} and the blowup formula \eqref{12} gives us 
\beq\label{16}
\bH_{2|n-2}=\bH_{1|n-1}+\sum_{k=0}^{n-4} \binom{n-2}{k} \bH_{2|k} (\bH_{\PP^{n-k-3}}-1).
\eeq
Repeating this way, we obtain the following. 
\begin{theo}\label{17} For $n>0$ and $r,s\ge 0$ with $r+s>0$, the following hold:
   \beq\label{18} \bH_{r|s} = \bH_{r-1|s+1} + \sum_{k=0}^{s-2} \binom{s}{k}\,\bH_{r|k}  \cdot (\bH_{\PP^{s-k-1}}-1), \quad r\ge 1,\eeq
   \beq\label{19} \bH_{\bcM_{g,n}}= \bH_{n|0} = \bH_{0|n} + \sum_{r=1}^{n-2} \sum_{k=0}^{s-2} \binom{s}{k}\, \bH_{r|k} \cdot (\bH_{\PP^{s-k-1}}-1).\eeq
All $\bH_{r|s}$ are recursively determined by \eqref{18} if we know $\bH_{0|n}=\bH_{\bcM_{g,(0^+)^n}}$ for all $n$.   
\end{theo}
In particular, the invariants $\bH_{\bcM_{g,n}}=\bH_{n|0}$ can be recursively calculated by \eqref{19} if we know the invariants $\bH_{\bcM_{g,(0^+)^n}}=\bH_{0|n}$ for all $n\ge 1$. 
Conversely, if we know $\bH_{\bcM_{g,n}}=\bH_{n|0}$ for all $n\ge 1$, we can recursively calculate $\bH_{\bcM_{g,(0^+)^n}}=\bH_{0|n}$ for all positive $n$.



The above factorization of the natural morphism
$$\rho:\bcM_{g,n}\lra \bcM_{g,(0^+)^n}$$
into smooth blowups in the proof of Theorem \ref{17} tells us that the Chow motive of $\bcM_{g,n}$ is a direct sum of Tate twists of the Chow motives of $\bcM_{g,(0^+)^m}$ for $m\le n$. 
Moreover, recall that the Hodge conjecture is about the surjectivity of the cycle class map $$A^p(X)_\QQ\lra H^{p,p}(X)\cap H^{2p}(X,\QQ)$$ and that 
we have
the blowup formulas for the Chow and ordinary cohomology groups \cite{Fulton, GrHa}. 
So the following are straightforward from the proof of Theorem \ref{17}. 
\begin{coro}\cite[Proposition 3.5, Corollary 3.6]{Ki}\label{HCc}

(1)    Let $g\ge 1$. Then the Hodge conjecture holds on $\bcM_{g,m}$ for all $m\le n$ if and only if it holds on $\bcM_{g,(0^+)^m}$ for all $m\le n$.

(2) The Hodge conjecture holds on $\bcM_{1,(0^+)^n}$ and on any $\bcM_{1,r|s}$ for $n=r+s\ge 1$. 
\end{coro}
Note that the Hodge conjecture holds on $\bcM_{1,n}$ for any $n\ge 1$ by 
\eqref{39}.

\bigskip

\section{Moduli spaces of minimally weighted curves}\label{S4} 
In this section, we let $g=1$ and let $\bH$ be the Hodge-Deligne polynomial. 
We find an explicit formula that calculates the Hodge-Deligne polynomials of the moduli spaces $\bcM_{1,(0^+)^n}$
of stable curves of genus 1 with $n$ marked points whose weights are $\le  1/n$. By Theorem \ref{17}, we can then calculate $\bH_{\bcM_{1,n}}$
for all $n>0$.

Note that in $\bcM_{1,(0^+)^n}$, all the marked points are allowed to coincide as long as they stay away from the nodes. 
Let 
$$\pi=\pi_n:\bcM_{1,(0^+)^n}\lra \bcM_{1,0^+}=\bcM_{1,1}$$
be the morphism that forgets all the marked points except the first. Using the stratification 
$$\bcM_{1,1}=\cM_{1,1}\sqcup \{\infty\}$$
where $\infty$ parameterizes the unique singular rational curve, we have the decomposition
$$\bcM_{1,(0^+)^n}=\cI_n\sqcup \cB_n, \quad \cB_n=\pi_n^{-1}(\infty), \ \ \cI_n=\pi_n^{-1}(\cM_{1,1})$$
into the interior and the boundary.

Let $I_n=\bH_{\cI_n}$ and $B_n=\bH_{\cB_n}$ be the Hodge-Deligne polynomials of $\cI_n$ and $\cB_n$ respectively so that 
$$\bH_{0|n}= \bH_{\bcM_{1,(0^+)^n}}=I_n+B_n $$
is the Hodge-Deligne polynomial of $\bcM_{1,(0^+)^n}$.

\subsection{The Hodge-Deligne Polynomial $I_n$ of $\cI_n$} 
In this subsection, we prove the following. 
\begin{prop}\label{3} $I_n(u,v)=I_n^{\mathrm{even}}(u,v)+I^{\mathrm{odd}}_n(u,v)$ with  
\[
I_n^{\mathrm{even}}(u,v) = 
\sum_{j=0}^{\lfloor\frac{n-1}{2}\rfloor} \binom{n-1}{2j} (uv+1)^{n-1-2j} \left( \alpha_{jj}(uv)^{j+1} - \sum_{i=0}^{j-1} \alpha_{ij}(uv)^i\right),
\]
\[I_n^{\mathrm{odd}}(u,v)=-\sum_{j=0}^{\lfloor\frac{n-1}{2}\rfloor} \binom{n-1}{2j} (uv+1)^{n-1-2j}  \sum_{i=0}^{j-1} \alpha_{ij}\beta_{2j-2i+2}(uv)^i(u^{2j-2i+1}+v^{2j-2i+1})  \]
where \beq\label{21} \alpha_{ij}=\binom{2j}{i}-\binom{2j}{i-1}
\and 
\beta_k=\left\{\begin{matrix}  \lfloor \frac{k}{12}\rfloor -1 & \text{if } k\equiv 2 \text{ mod } 12\\
\lfloor \frac{k}{12}\rfloor & \text{if } k\not\equiv 2 \text{ mod } 12 .\end{matrix}\right.
\eeq
\end{prop}
See Example \ref{5} for an explicit computation of $I_n$ for $n\le 10$. 

\medskip

Let $q:\cC\to \cM_{1,1}$ denote the universal curve.
By our choice of stability condition, all the marked points are allowed to coincide and hence $\cI_n$ is nothing but the $(n-1)$-fold fiber product 
$$\pi:\cI_n=\cC\times_{\cM_{1,1}}\cC\times_{\cM_{1,1}}\cdots\times_{\cM_{1,1}}\cC\lra \cM_{1,1}$$
which is a smooth projective fibration with fiber $E^{n-1}$ for an elliptic curve $E\in \cM_{1,1}$.  

Let $V=R^1q_*\QQ$. Then $V$ is a local system of rank 2 over $\cM_{1,1}$ and $$Rq_*\QQ\cong (\QQ\oplus\QQ(-1)[-2])\oplus V[-1].$$ 
Let $\pi'=\pi|_{\cI_n}$. By the relative K\"unneth theorem, we have 
\beq\label{25}
R\pi'_*\QQ_{\cI_n}\cong ((\QQ\oplus \QQ(-1)[-2])\oplus V[-1])^{\otimes n-1}
=\sum_{k=0}^{n-1} \binom{n-1}{k} C_k\otimes V^{\otimes k}[-k]
\eeq
where $C_k=(\QQ\oplus \QQ(-1)[-2])^{\otimes n-1-k}$.

Since $\cM_{1,1}\cong SL_2(\ZZ)\backslash \HH$ where $\HH$ denotes the upper half plane and the generic stabilizer $-I\in SL_2(\ZZ)$ acts on $V^{\otimes k}$ as $(-1)^k$, we have the vanishing 
$$H^*_c(\cM_{1,1},V^{\otimes k})=0, \quad \forall \text{odd } k.$$  
Hence \eqref{25} gives us
\beq\label{26}
H^*_c(\cI_n)=H^*_c(\cM_{1,1},R\pi'_*\QQ_{\cI_n})
\cong \sum_{j=0}^{\lfloor \frac{n-1}{2}\rfloor} \binom{n-1}{2j}  C_{2j}\otimes H^*_c(\cM_{1,1},V^{\otimes 2j})[-2j].\eeq

By the Clebsch-Gordan rule and \eqref{21}, we have 
$$V^{\otimes 2j}\cong \bigoplus_{i=0}^{j-1} \alpha_{ij}\, \mathrm{Sym}^{2j-2i} V(-i)\oplus \alpha_{jj}\QQ.$$ 
Therefore \eqref{26} gives us 
\beq\label{27}
H^*_c(\cI_n)\cong 
\sum_{j=0}^{\lfloor \frac{n-1}{2}\rfloor} \binom{n-1}{2j}  C_{2j}\otimes 
\left( \alpha_{jj}H^*_c(\cM_{1,1},\QQ)[-2j]+\sum_{i=0}^{j-1} \alpha_{ij} H^*_c(\cM_{1,1},V_{2j-2i})[-2i]\right).\eeq
Since $H^*_c(\cM_{1,1},\QQ)\cong \QQ(-1)[-2]$, Proposition \ref{3} follows immediately from the following.

\begin{lemm}\label{35} For $k>0$, we have an isomorphism
    $$H^*_c(\cM_{1,1},V_{2k})=H^1_c(\cM_{1,1},V_{2k})\cong \QQ\oplus (S_{2k+2}\oplus \bar S_{2k+2})$$
where $S_{2k+2}$ denotes the space of cusp forms of weight $2k+2$ with respect to $SL_2(\ZZ)$. The first factor $\QQ$ has weight $0$ and  $S_{2k+2}$ is of Hodge type $(2k+1,0)$. 
\end{lemm}
\begin{proof}
    Let $\jmath:\cM_{1,1}\to \bcM_{1,1}$ denote the inclusion map. Then we have an exact triangle 
    $$\jmath_!V_{2k}\lra \jmath_*V_{2k}\lra \jmath_*V_{2k}|_\infty\lra .$$
    By the Eichler-Shimura isomorphism \cite[Theorem 8.4]{Shi}, we have 
    $$H^*(\bcM_{1,1},\jmath_*V_{2k})=H^1(\bcM_{1,1},\jmath_*V_{2k})\cong H^1_{\mathrm{par}}(SL_2(\ZZ),\mathrm{Sym}^{2k}(\CC^2))\cong S_{2k+2}\oplus \bar S_{2k+2}$$ where $S_{2k+2}$ has Hodge type $(2k+1,0)$. 
    We thus find that $H^*_c(\cM_{1,1},V_{2k})=H^1_c(\cM_{1,1},V_{2k})$ and have an exact sequence
\[    0\lra H^0(\infty,\jmath_*V_{2k}|_\infty)\lra H^1_c(\cM_{1,1},V_{2k})\lra S_{2k+2}\oplus \bar S_{2k+2} \lra 0 .\]
As the cusp $\infty$ has the stabilizer group generated by
\[\sigma:=\left( \begin{matrix}
1&1\\ 0&1    
\end{matrix}\right),\]
$H^0(\infty,\jmath_*V_{2k}|_\infty)$ is the $\sigma$-invariant part $(\mathrm{Sym}^{2k}\QQ^2)^\sigma\cong \QQ$. This completes the proof. 
\end{proof}
It is well known that $\dim S_{2k+2}=\beta_{2k+2}$ for $k>0$ by Riemann-Roch, using the notation of \eqref{21}.  

\medskip

\subsection{The Hodge-Deligne Polynomial $B_n$ of $\cB_n$}

Let $\gamma_{n,1}=1$, $\gamma_{n,2}=2^{n-1}-1$ and for $k>2$, let \beq\label{20} \gamma_{n,k} =\left\{\begin{matrix} n\\ k\end{matrix}\right\}\frac{(k-1)!}{2} = \frac{1}{2k} \sum_{i=0}^k (-1)^{k-i} \binom{k}{i} i^n\eeq
where $\left\{\begin{matrix} n\\ k\end{matrix}\right\}$ denotes the Stirling number of the second kind, i.e. the number of ways to partition $[n]=\{1,2,\cdots,n\}$ into $k$ unordered subsets. 
In this subsection, we prove the following. 
\begin{prop}\label{4} 
With $\gamma_{n,k}$ in \eqref{20}, 
\[
B_n(u,v) = \frac{(uv+1)^{n-1}+(uv-1)^{n-1}}{2} + \gamma_{n,2} \frac{(uv+1)^{n-2}+(uv-1)^{n-2}}{2} + \sum_{j=3}^n \gamma_{n,j} (uv-1)^{n-j}
\]
\end{prop}
See Example \ref{6} for the calculation of $B_n$ for $n\le 10$. 

\medskip

Observe that we have a decomposition $$\cB_n=\bigsqcup_{1\le j\le n}\cB_{n,j}$$ where 
$\cB_{n,j}$ denotes the locus of curves with $j$ rational components i.e. the dual graph is a necklace with $j$ vertices.

If $(C,p_1,\cdots, p_n)\in \cB_{n,1}$, $C$ is a rational curve with  normalization $\PP^1$. Let $0,\infty\in \PP^1$ map to the node of $C$. Then 
we find that 
$$\cB_{n,1}\cong (\CC^*)^n/(\CC^*\rtimes S_2)\cong (\CC^*)^{n-1}/S_2$$
where $-1\in S_2$ acts by $(z_1,\cdots,z_{n-1})\mapsto (z_1^{-1},\cdots,z_{n-1}^{-1}).$
Therefore, we have $$H^*_c(\cB_{n,1})\cong (H^*_c(\CC^*)^{\otimes n-1})^{S_2}$$
whose Hodge-Deligne polynomial is 
\beq\label{28}
\bH_{\cB_{n,1}}(u,v)=\frac12\left( (uv+1)^{n-1} + (uv-1)^{n-1} \right).\eeq

If $(C,p_1,\cdots,p_n)\in \cB_{n,2}$, $C=C_1\cup C_2$ with $C_1,C_2\cong \PP^1$.
The locus $\cB_{n,2,k}$ of such pointed curves where one component has $k$ marked points while the other has $n-k$ points is 
$$((\CC^*)^k/\CC^*\times (\CC^*)^{n-k}/\CC^*)/S_2\cong (\CC^*)^{n-2}/S_2$$
where $S_2$ interchanges the two nodes. We thus find that 
\beq\label{29}
\bH_{\cB_{n,2}}(u,v)=(2^{n-1}-1) \frac12\left( (uv+1)^{n-2} + (uv-1)^{n-2} \right).\eeq

For $j\ge 3$, the number of ways to distribute $n$ marked points to $j$ components is $\gamma_{n,j}$ in \eqref{20}. For each such distribution, we have a connected component of $\cB_{n,j}$ isomorphic to $(\CC^*)^{n-j}$. Hence we find that 
\beq\label{30}
\bH_{\cB_{n,j}}(u,v) = \gamma_{n,j}(uv-1)^{n-j},\quad \text{for } j\ge 3.\eeq

Adding \eqref{28}, \eqref{29} and \eqref{30}, we obtain Proposition \ref{4}. 
 
\medskip

\subsection{The Hodge-Deligne polynomial $\bH_{\bcM_{1,(0^+)^n}}$} 

By Propositions \ref{3} and \ref{4}, we obtain the following.
\begin{theo}\label{22}
The Hodge-Deligne polynomial $\bH_{\bcM_{1,(0^+)^n}}=\bH_{0|n}$ of $\bcM_{1,(0^+)^n}$ is 
\[ \bH_{0|n}(u,v)= \sum_{j=0}^{\lfloor\frac{n-1}{2}\rfloor} \binom{n-1}{2j} (uv+1)^{n-1-2j} \left( \alpha_{jj}(uv)^{j+1} - \sum_{i=0}^{j-1} \alpha_{ij}(uv)^i\right)
\]
\[ +\frac{(uv+1)^{n-1}+(uv-1)^{n-1}}{2} + \gamma_{n,2} \frac{(uv+1)^{n-2}+(uv-1)^{n-2}}{2} + \sum_{j=3}^n \gamma_{n,j} (uv-1)^{n-j} \]
\[ -\sum_{j=0}^{\lfloor\frac{n-1}{2}\rfloor} \binom{n-1}{2j} (uv+1)^{n-1-2j}  \sum_{i=0}^{j-1} \alpha_{ij}\beta_{2j-2i+2}(uv)^i(u^{2j-2i+1}+v^{2j-2i+1}) \]
where $\alpha_{ij}$, $\beta_k$ and $\gamma_{n,k}$ are constants defined by \eqref{21} and \eqref{20} respectively. 
\end{theo}


See Example \ref{7} for a computation of $\bH_{0|n}$ for $n\le 10$.

\medskip

By Theorems \ref{22} and \ref{17}, we can now calculate $\bH_{\bcM_{1,n}}(u,v)$ recursively for all $n$. By \cite{Pet}, the cycle class map induces an isomorphism 
\beq\label{39} 
R^*(\bcM_{1,n})\cong H^{2*}(\bcM_{1,n})\eeq 
where $R^*(\bcM_{1,n})$ denotes the tautological ring and hence we can also compute 
$\dim R^i(\bcM_{1,n})$ as the coefficient of $(uv)^i$ in $\bH_{\bcM_{1,n}}(u,v)$. In particular, the Poincar\'e polynomial of the tautological ring is the same as the even degree Poincar\'e polynomial $P^{\mathrm{even}}_{\bcM_{1,n}}(x)$.

\bigskip

\section{Computations for $n\le 10$}
Letting $x=uv$, by Proposition \ref{3}, we can evaluate $I_n$ as follows. 
\begin{exam}\label{5}
\begin{align*}
I_1(x) &= x \\
I_2(x) &= x^2 + x \\
I_3(x) &= x^3 + 3x^2 + x - 1 \\
I_4(x) &= x^4 + 6x^3 + 6x^2 - 2x - 3 \\
I_5(x) &= x^5 + 10x^4 + 20x^3 + 4x^2 - 14x - 7 \\
I_6(x) &= x^6 + 15x^5 + 50x^4 + 40x^3 - 30x^2 - 49x - 15 \\
I_7(x) &= x^7 + 21x^6 + 105x^5 + 160x^4 - 183x^2 - 139x - 31 \\
I_8(x) &= x^8 + 28x^7 + 196x^6 + 469x^5 + 280x^4 - 427x^3 - 700x^2 - 356x - 63 \\
I_9(x) &= x^9 + 36x^8 + 336x^7 + 1148x^6 + 1386x^5 - 406x^4 - 2436x^3 - 2224x^2 - 860x - 127 \\
I_{10}(x) &= x^{10} + 45x^9 + 540x^8 + 2484x^7 + 4662x^6 + 1764x^5 - 6090x^4 - 9804x^3 - 6372x^2 - 2003x - 255
\end{align*}
\end{exam}

\medskip

By Proposition \ref{4}, we can evaluate $B_n$ as follows with $x=uv$. 
\begin{exam}\label{6}
\begin{align*}
B_1(x) &= 1 \\
B_2(x) &= x + 1 \\
B_3(x) &= x^2 + 3x + 2 \\
B_4(x) &= x^3 + 7x^2 + 9x + 4 \\
B_5(x) &= x^4 + 15x^3 + 31x^2 + 25x + 8 \\
B_6(x) &= x^5 + 31x^4 + 100x^3 + 111x^2 + 65x + 16 \\
B_7(x) &= x^6 + 63x^5 + 316x^4 + 476x^3 + 351x^2 + 161x + 32 \\
B_8(x) &= x^7 + 127x^6 + 987x^5 + 2178x^4 + 1883x^3 + 1023x^2 + 385x + 64 \\
B_9(x) &= x^8 + 255x^7 + 3053x^6 + 10515x^5 + 12307x^4 + 6637x^3 + 2815x^2 + 897x + 128 \\
B_{10}(x) &= x^9 + 511x^8 + 9366x^7 + 51313x^6 + 92466x^5 + 62065x^4 + 21654x^3 + 7423x^2 + 2049x + 256
\end{align*}
\end{exam}

\medskip 

By adding $I_n$ and $B_n$, we find $T_n=\bH_{\bcM_{1,0|n}}$ as follows. 
\begin{exam}\label{7}
\begin{align*}
T_1(x) &= x + 1 \\
T_2(x) &= x^2 + 2x + 1 \\
T_3(x) &= x^3 + 4x^2 + 4x + 1 \\
T_4(x) &= x^4 + 7x^3 + 13x^2 + 7x + 1 \\
T_5(x) &= x^5 + 11x^4 + 35x^3 + 35x^2 + 11x + 1 \\
T_6(x) &= x^6 + 16x^5 + 81x^4 + 140x^3 + 81x^2 + 16x + 1 \\
T_7(x) &= x^7 + 22x^6 + 168x^5 + 476x^4 + 476x^3 + 168x^2 + 22x + 1 \\
T_8(x) &= x^8 + 29x^7 + 323x^6 + 1456x^5 + 2458x^4 + 1456x^3 + 323x^2 + 29x + 1 \\
T_9(x) &= x^9 + 37x^8 + 591x^7 + 4201x^6 + 11901x^5 + 11901x^4 + 4201x^3 + 591x^2 + 37x + 1 \\
T_{10}(x) &= x^{10} + 46x^9 + 1051x^8 + 11850x^7 + 55975x^6 + 94230x^5 + 55975x^4 + 11850x^3 + 1051x^2 + 46x + 1
\end{align*}
\end{exam}
Note that this matches perfectly with Table 2 in \cite{KSY} with $x=uv$.

\medskip

By \eqref{18} and Theorem \ref{22}, we can compute $\bH_{r|s}=\bH_{\bcM_{1,r|s}}$ as follows with $x=uv$. 
\begin{exam}\label{23}
\begin{align*}
\bH_{0|1} = \bH_{1|0} &= x + 1 \\
\bH_{0|2} = \bH_{1|1} = \bH_{2|0} &= x^2 + 2x + 1 \\
\bH_{0|3} &= x^3 + 4x^2 + 4x + 1 \\
\bH_{1|2} = \bH_{2|1} = \bH_{3|0} &= x^3 + 5x^2 + 5x + 1 \\
\bH_{0|4} &= x^4 + 7x^3 + 13x^2 + 7x + 1 \\
\bH_{1|3} &= x^4 + 11x^3 + 21x^2 + 11x + 1 \\
\bH_{2|2} = \bH_{3|1} = \bH_{4|0} &= x^4 + 12x^3 + 23x^2 + 12x + 1 \\
\bH_{0|5} &= x^5 + 11x^4 + 35x^3 + 35x^2 + 11x + 1 \\
\bH_{1|4} &= x^5 + 22x^4 + 79x^3 + 79x^2 + 22x + 1 \\
\bH_{2|3} &= x^5 + 26x^4 + 97x^3 + 97x^2 + 26x + 1 \\
\bH_{3|2} = \bH_{4|1} = \bH_{5|0} &= x^5 + 27x^4 + 102x^3 + 102x^2 + 27x + 1.
\end{align*}
\end{exam}

Note that this matches perfectly with Table 1 in \cite{KSY} after applying the hook-length formula to extract the Hodge-Deligne polynomials.  

\medskip

By \eqref{19} and Theorem \ref{22}, we can compute $\bH_{n|0}=\bH_{\bcM_{1,n}}$ as follows with $x=uv$. 
\begin{exam}\label{24}
\begin{align*}
\bH_{\bcM_{1,1}} &= x + 1 \\
\bH_{\bcM_{1,2}} &= x^2 + 2x + 1 \\
\bH_{\bcM_{1,3}} &= x^3 + 5x^2 + 5x + 1 \\
\bH_{\bcM_{1,4}} &= x^4 + 12x^3 + 23x^2 + 12x + 1 \\
\bH_{\bcM_{1,5}} &= x^5 + 27x^4 + 102x^3 + 102x^2 + 27x + 1 \\
\bH_{\bcM_{1,6}} &= x^6 + 58x^5 + 421x^4 + 756x^3 + 421x^2 + 58x + 1 \\
\bH_{\bcM_{1,7}} &= x^7 + 121x^6 + 1612x^5 + 5077x^4 + 5077x^3 + 1612x^2 + 121x + 1 \\
\bH_{\bcM_{1,8}} &= x^8 + 248x^7 + 5802x^6 + 31072x^5 + 52402x^4 + 31072x^3 + 5802x^2 + 248x + 1 \\
\bH_{\bcM_{1,9}} &= x^9 + 503x^8 + 19925x^7 + 175036x^6 + 480097x^5 + 480097x^4 + 175036x^3 + 19925x^2 + 503x + 1 \\
\bH_{\bcM_{1,10}} &= x^{10} + 1014x^9 + 66090x^8 + 920263x^7 + 3975949x^6 + 6349238x^5 + 3975949x^4 + 920263x^3 + \cdots.
\end{align*}
\end{exam}
Note that this matches perfectly with \cite[eM1bar.txt]{BeT} and \cite[page 491]{Getzler2}. 

\medskip

By Theorems \ref{17} and \ref{22}, it is straightforward to write up a computer program that computes the Hodge-Deligne and Poincar\'e polynomials of $\bcM_{1,n}$ for higher $n$, quite efficiently. For $n=50$, we get the following picture for the probability distribution \beq\label{31}
p_{n,i}=\frac{\dim H^{2i}(\bcM_{1,n})}{\sum_{k=0}^n\dim H^{2k}(\bcM_{1,n})}.\eeq
The red curve in the picture is the graph of a Gaussian function. 

\begin{center}
\includegraphics[width=10cm]{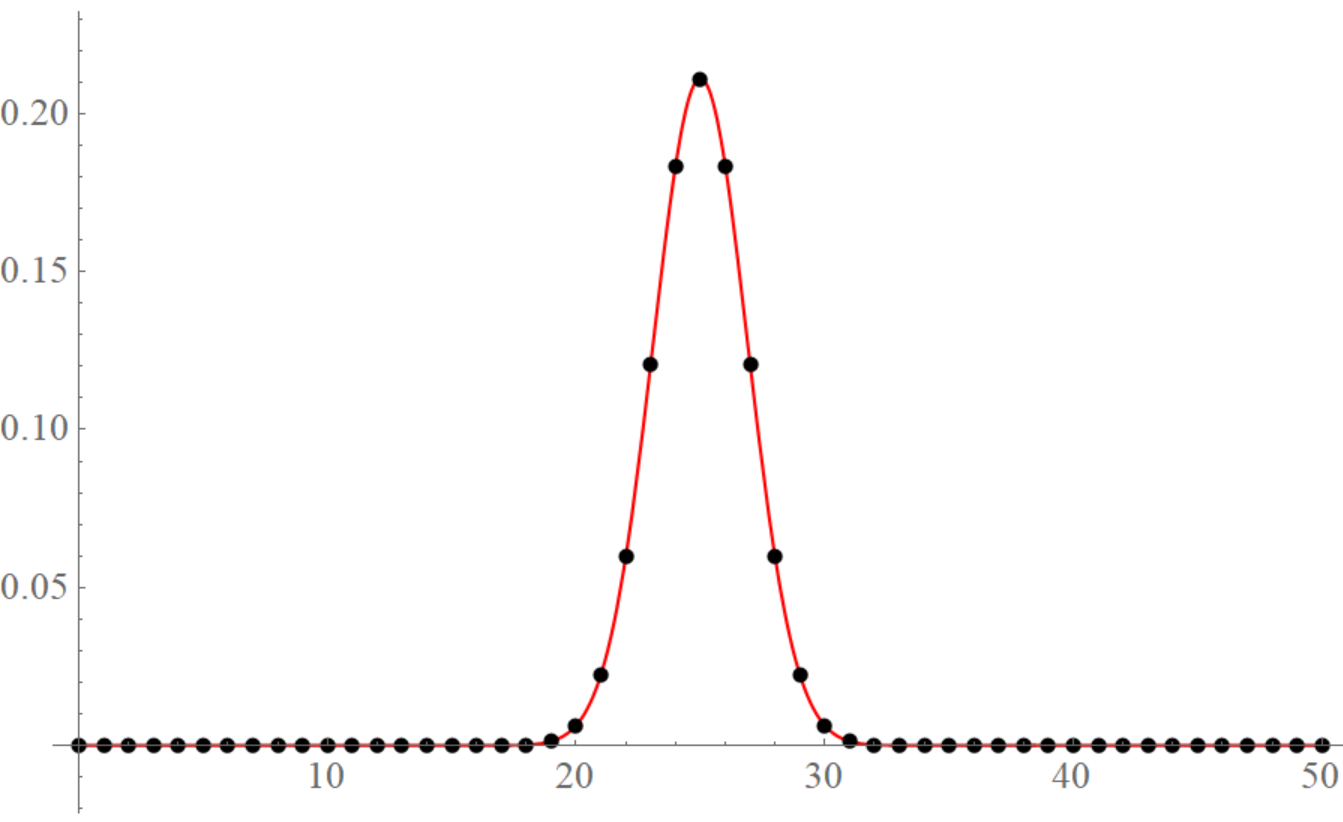}
\end{center}

In the subsequent paper \cite{CKga}, we will prove that indeed the probability distribution \eqref{31} is asymptotically Gaussian. 
Furthermore, we will prove a formula that gives us effective estimates of the Betti numbers for $n$ large and show that the Betti sequence is asymptotically 3-ultra log concave. 

\medskip

\noindent\textbf{Update on Question \ref{Q1.4}}. 
With the formulas of this paper, the Betti numbers of $\bcM_{1,n}$ can be computed
easily for large $n$, and one can experiment with various properties of the Betti
sequence.

After posting this paper in the arXiv, Gergely B\'erczi informed me that the even degree Poincar\'e polynomial (or the Poincar\'e polynomial of the tautological ring) of $\bcM_{1,n}$ is not real-rooted when $n=8706$. Still it seems reasonable to expect that the even degree Betti sequence is ultra-log-concave (ULC) based on numerical computations. 

The reduction to minimally weighted curves was made $S_n$-equivariant in
\cite[Theorem 3.9, Remark 3.10]{Ki}, and together with an $S_n$-equivariant
computation for $\bcM_{1,(0^+)^n}$, this gives a new proof of Getzler's formula,
answering Question \ref{Q1.4}(1) \cite{CKr}. In \cite{Ki}, the strategy of this
paper is applied to compute the $S_n$-equivariant Hodge-Deligne polynomials of
$\bcM_{2,n}$, which answers Question \ref{Q1.4}(5) for $g=2$.

On the other hand, in \cite{CKga}, it was proved for all $g\ge 1$ that the even degree Betti sequence of $\bcM_{g,n}$ is asymptotically 3-ULC but not 4-ULC. Interestingly, the even degree Betti sequence of $\bcM_{g,(0^+)^n}$ is also asymptotically normally distributed with slightly larger variance and asymptotically 3-ULC.

\bigskip

\bibliographystyle{amsplain}

\end{document}